\documentclass[10pt, a4paper]{article}
\usepackage{indentfirst}
\usepackage{amsmath, amssymb, amsfonts, amsthm, graphicx}
\usepackage{geometry}
\usepackage{setspace}
\usepackage{algorithm}
\usepackage{algpseudocode}
\makeatletter
\usepackage{caption}
\usepackage{subfig}
\usepackage[section]{placeins}
\usepackage{float}
\usepackage{epstopdf}
\usepackage{cite}
\usepackage{xcolor}
\newtheorem{theorem}{Theorem}[section]

\newtheorem{lemma}[theorem]{Lemma}
\newtheorem{claim}{Claim}
\newtheorem{corollary}[theorem]{Corollary}

\newtheorem{remark}{Remark}

\newtheorem{conjecture}{Conjecture}[section]
\newtheorem{problem}[conjecture]{Problem}
\def\qed{\hfill \rule{4pt}{7pt}}

\begin{document}
\title{\Large\bf From Halin's Edge Removability to Matching Removability in $k$-Connected Graphs}

\author{Hengzhe Li$^{a}$, Mingming Zhou$^{a}$, Shinya Fujita$^{b}$\footnote{Corresponding author. Email: shinya.fujita.ph.d@gmail.com}, Yaping Mao$^{c}$\\
\small $^{a}$College of Mathematics and Statistics,\\
\small Henan Normal University, Xinxiang 453007, P.R. China\\
\small $^{b}$International College of Arts and Sciences, Yokohama City University\\
\small 22-2 Seto, Kanazawa-ku, Yokohama 236-0027, Japan\\
\small $^{c}$School of Mathematics and Statistics,\\
		\small Qinghai Normal University, Xining, Qinghai 810008, China.\\
\small Emails: lihengzhe@htu.edu.cn, zhoumming1218@163.com, shinya.fujita.ph.d@gmail.com, yapingmao@outlook.com}
\date{}
\maketitle
\begin{abstract}
We study matching-removability under the degree/connectivity regime of Halin's theorem, which asserts that every $k$-connected graph $G$ with minimum degree $\delta(G)\ge k+1$ contains an edge $e$ such that $G-e$ remains $k$-connected. For $k,\ell\ge 1$, an $\ell$-matching is a matching of size $\ell$. A matching $M$ in a $k$-connected graph $G$ is {\it $k$-removable} if $G-M$ remains $k$-connected. We improve Halin's result by proving that every $k$-connected graph $G$ with $\delta(G)\ge k+1$ contains a $k$-removable $2$-matching, except when $k=1$ and $G$ is a cycle. For small $k$ we obtain stronger bounds: (i) $k=1$: a 1-removable $\min\{\lfloor
n/2\rfloor,\delta(G)\}$-matching; (ii) $k=2$: a 2-removable $\lceil(\delta(G)+1)/2\rceil$-matching, with a unique tight exception when $\delta(G)$ is even and $G\cong K_{\delta(G)+1}$; and (iii) $k=3$: for $\delta(G)\ge 5$, a $3$-removable $\lceil(\delta(G)+1)/2\rceil$-matching.

We also show that for every $n\ge 2\delta$, there exists a $k$-connected $n$-vertex graph $G$ with minimum degree $\delta$ that does not contain a $k$-removable matching of size at least $\delta(G)+1$.
Moreover, for $k\le 2$ there exists a $k$-removable $(\delta(G)-c)$-matching for some $c\le 3$, which is optimal up to the additive constant.

\end{abstract}

\section{Introduction}

Throughout, graphs are finite, undirected, and simple. Undefined terminology follows Bondy and Murty~\cite{Bondy2007}. For a graph $G$, we denote its vertex set, edge set, connectivity, and minimum degree by $V(G)$, $E(G)$, $\kappa(G)$, and $\delta(G)$, respectively; let $n=|V(G)|$.

For a vertex subset $S \subseteq V(G)$,
we denote
the {\it induced subgraph} of $G$ by $G[S]$.  For an edge subset $F \subseteq E(G)$, the {\it edge-induced subgraph} of $G$ is denoted by $G[F]$, where it is the subgraph of $G$ with the vertex set consisting of all vertices incident to edges in $F$ together with the edge set $F$.

For a vertex subset $S \subseteq V(G)$, let $G - S := G[V(G) \setminus S]$. For an edge subset $E' \subseteq E(G)$, $G - E'$ denotes the graph with vertex set $V(G)$ and edge set $E(G) \setminus E'$. We write $G - v$ and $G - e$ as abbreviations for $G - \{v\}$ and $G - \{e\}$, respectively.

Let $G$  be a connected graph.
 A {\it cut vertex} of $G$ is a vertex $v$ such that $G-v$ is a disconnected graph. A {\it bridge} (or cut edge) is an edge $e$ such that $G-e$ is a disconnected graph. For an integer $k\ge 1$, a vertex subset $S\subseteq V(G)$ on $k$ vertices is a {\it $k$-separating set} if $G-S$ is disconnected. A graph $G$ is {\it $k$-connected} if it has no $(k-1)$-separating set; it is {\it $k$-edge-connected} if $G-F$ is connected for every $F\subseteq E(G)$ with $|F|\le k-1$.

For $k \ge 1$, a vertex set $S$ in a $k$-connected graph $G$ is called {\it $k$-removable} if $G - S$ remains $k$-connected. Every vertex in a $k$-removable vertex set is a {\it $k$-removable vertex}. An edge set $F$ in a $k$-connected graph $G$ is called {\it $k$-removable} if $G - F$ remains $k$-connected. Every edge in a $k$-removable edge set is a {\it $k$-removable edge}. If $k$-removable edge set $F$ is a matching, then $F$ is a {\it $k$-removable matching}.

Halin~\cite{Halin1969} initiated the study of edge-removability under the condition $\delta(G)\ge k+1$.

\begin{theorem}[Halin~\cite{Halin1969}]\label{thm:halin}
Every $k$-connected graph $G$ with $\delta(G)\ge k+1$ contains an edge $e$ such that $G-e$ remains $k$-connected.
\end{theorem}

The minimum degree condition in Theorem~\ref{thm:halin} is best possible, since a $k$-connected $k$-regular graph contains no $k$-removable edge. From this point of view, in this paper, we will mainly discuss the edge-removability of a $k$-connected graph $G$ such that $\delta(G)\geq k+1$.

On the other hand, vertex-removability in $k$-connected graphs with higher minimum degree was subsequently investigated by Chartrand, Kaugars, and Lick~\cite{Chartrand1972}.

\begin{theorem}[Chartrand--Kaugars--Lick~\cite{Chartrand1972}]\label{thm:ckl}
Every $k$-connected graph $G$ with $\delta(G)\ge \left\lfloor 3k/2\right\rfloor$ has a vertex $x$ such that $G-x$ is still $k$-connected.
\end{theorem}

Fujita and Kawarabayashi~\cite{Fujita2008} strengthened this direction by proving that every $k$-connected graph $G$ with $\delta(G)\ge \left\lfloor 3k/2\right\rfloor+2$ has an edge $xy$ such that $G-x-y$ remains $k$-connected, and they proposed a broader conjecture on removable connected subgraphs: for each fixed $m$, there exists a function $f(m)$ such that if $G$ is $k$-connected with $\delta(G)\ge 3k/2+f(m)$, then $G$ contains a connected subgraph $W$ of order $m$ such that $G-W$ still $k$-connected. Mader~\cite{Mader2010} confirmed the existence of such a function via the following path-removability theorem.

\begin{theorem}[Mader~\cite{Mader2010}]\label{thm:mader-path}
Let $k,m$ be positive integers. Every $k$-connected graph $G$ with $\delta(G)\ge \left\lfloor 3k/2\right\rfloor + m - 1$ contains a path $P$ of order $m$ such that $G-V(P)$ remains $k$-connected.
\end{theorem}

Furthermore, Mader pointed out the possibility of extending Theorem~\ref{thm:mader-path} from paths to arbitrary trees.

\begin{conjecture}[Mader~\cite{Mader2010}]\label{conj:mader-tree}
For any $k\ge 1$ and any tree $T$ of order $m$, every $k$-connected graph $G$ with $\delta(G)\ge \left\lfloor 3k/2\right\rfloor + m - 1$ contains a subtree $T'\cong T$ such that $G-V(T')$ is $k$-connected.
\end{conjecture}

Mader proved that Conjecture~\ref{conj:mader-tree} is true when $T$ is a path~\cite{Mader2010}. Diwan and Tholiya~\cite{Diwan2009} settled the case $k=1$ (motivated by Locke's conjecture). For particular families of trees or for $2$-connected graphs, the conjecture was established by Hasunuma and Ono~\cite{Hasunuma2020}, Hasunuma~\cite{Hasunuma2021}, Lv and Zhang~\cite{Lv2020}, Tian, Meng, Lai, and Xu~\cite{Tian2018}, Tian, Lai, Xu, and Meng~\cite{Tian2019}. Hong and Liu~\cite{Hong2022} verified the conjecture for $k\le 3$; the range $k\ge 4$ remains open. For further advances regarding Mader's Conjecture, readers may consult the recent survey by Tian and Meng~\cite{Tian2026}. In addition, Fujita~\cite{Fujita2025} explored $k$-removable paths in $k$-connected triangle-free graphs.

Hasunuma proposed an edge-removable analogue in which one deletes the edges of an embedded tree while keeping high connectivity in  $k$-(edge)-connected graphs.

\begin{conjecture}[Hasunuma~\cite{Hasunuma2023}]\label{conj:hasunuma-edge}
For $k\ge 1$ and any tree $T$ of order $m$, every $k$-connected (or $k$-edge-connected) graph $G$ with $\delta(G)\ge k+m-1$ contains a subtree $T'\cong T$ such that $G-E(T')$ remains $k$-connected (or $k$-edge-connected, respectively).
\end{conjecture}

Hasunuma verified Conjecture~\ref{conj:hasunuma-edge} for $k\le 2$, and  Liu, Liu, and Hong~\cite{Liu2023} extended this to $k=3$. For more advances on Hasunuma's Conjecture, we refer readers to the survey by Tian and Meng~\cite{Tian2026}.
For $k\ge 4$, the conjecture is still open. Moreover, for $k\le 2$, Hasunuma obtained the following refinement.

\begin{theorem}[Hasunuma~\cite{Hasunuma2023}]\label{thm:hasunuma-kle2}
Let $k\le 2$ and $T$ be a tree of order $m$. Every $k$-connected (or $k$-edge-connected) graph $G$ with $\delta(G)\ge \max\{\Delta(T)+k,\, m-1\}$ contains a subtree $T'\cong T$ such that $G-E(T')$ remains $k$-connected (or $k$-edge-connected).
\end{theorem}

The above results concern the removability of connected subgraphs. In this paper, we relax the connectivity requirement on the deleted subgraph and ask: can one remove a {\it matching} while preserving $k$-connectivity?

By Theorem~\ref{thm:hasunuma-kle2}, we obtain the following immediate corollary, which strengthens Halin's theorem for $k\le 2$ when $\delta(G)\ge k+2$.

For $\ell\ge 1$, an $\ell$-matching is a matching of size $\ell$.

\begin{corollary}\label{cor:hasunuma-matching}
For $k\le 2$, every connected graph $G$ with $\delta(G)\ge k+2$ contains a $k$-removable $\lfloor (\delta(G)+1)/2\rfloor$-matching.
\end{corollary}

For general $k\ge 1$, we advance Halin's theorem from single edge to matchings and prove the following.

\begin{theorem}\label{thm:two-matching}
For $k\ge 1$, every $k$-connected graph $G$ with $\delta(G)\ge k+1$ contains a $k$-removable $2$-matching, unless $k=1$ and $G$ is a cycle.
\end{theorem}

For $k \le 3$, we can increase the size of a $k$-removable matching as follows.

\begin{theorem}\label{thm:half-delta}
(i) If $k\le 2$, then every $k$-connected graph $G$ with $\delta(G)\ge k+1$ contains a $k$-removable $\lceil (\delta(G)+1)/2\rceil$-matching, unless $G$ is the complete graph with an odd number of vertices.\\
(ii) Every $3$-connected graph $G$ with $\delta(G)\ge 5$ contains a $3$-removable $\lceil (\delta(G)+1)/2\rceil$-matching.
\end{theorem}

For a $k$-connected graph $G$ with $k\ge 4$, we need a stronger lower bound on $\delta(G)$ to construct $k$-removable $\lceil(\delta(G)+1)/2\rceil$-matching.

\begin{theorem}\label{thm:3k-1}
For $k \ge 4$, every $k$-connected graph $G$ with $\delta(G) \ge 3k-1$ contains a $k$-removable $\lceil(\delta(G)+1)/2\rceil$-matching, unless
$G$ is the complete graph with an odd number of vertices.
\end{theorem}

For $1$-connected graphs, we further sharpen the bound as follows.

\begin{theorem}\label{thm:half-n-min}
Let $G$ be a connected graph on $n$ vertices with $\delta(G)\ge 3$. Then $G$ contains a $1$-removable $\min\{\lfloor n/2\rfloor,\, \delta(G)\}$-matching.
\end{theorem}

One might guess that, for a $k$-connected graph $G$ with $\delta(G)\ge k+1$, larger order forces a larger $k$-removable matching. However, the following example implies that the answer is negative. Let $k,\delta$ be two positive integers such that $\delta=k+1$.
The complete bipartite graph $K_{\delta,\,n-\delta}$ with $n\geq 2\delta$ satisfies $\delta(K_{\delta,\,n-\delta})=\kappa(K_{\delta,\,n-\delta})=\delta>k$ and $\alpha'(K_{\delta,\,n-\delta})=\delta$, where $\alpha'(G)$ denotes the maximum size of a matching in a graph $G$. Obviously, it has no $k$-removable $(\delta+1)$-matching.

Motivated by the above properties of $K_{\delta,\,n-\delta}$, we now define the following function concerning the size of $k$-removable matching in $k$-connected graphs $G$ such that $|V(G)| \ge 2\delta$ and $\delta(G)\ge \delta$.
Let $\delta$ and $k$ be two integers with $\delta > k \ge 1$.
Define $f(k,\delta)$ as the largest integer such that every $k$-connected graph $G$ with $|V(G)| \ge 2\delta$ and $\delta(G)\ge \delta$ has a $k$-removable $f(k,\delta)$-matching. The previous example shows that $f(k,\delta) \le \delta$. We establish the following exactly values and bounds for $f(k,\delta)$.

\begin{theorem}\label{thm:f-k-delta}
Let $\delta$ and $k$ be integers with $\delta > k \ge 1$. The following statements hold:
\begin{itemize}
\item[(i)] $(\delta+1)/2\le f(k,\delta)\le \delta$ for $\delta\ge 3k-1$.
\item[(ii)] $f(1,\delta)=1$ if $\delta=2$, and $f(1,\delta)=\delta$ if $\delta\ge 3$.
\item[(iii)] $2\lfloor(\delta-2)/2\rfloor\le f(2,\delta)\le \delta$.
\item[(iv)] $2\le f(k,k+1)\le k$ for $k\ge 2$; in particular, $f(2,3)=2$.
\end{itemize}
\end{theorem}


\paragraph{Organization.}
Section~\ref{sec:prelim} collects preliminaries. Section~\ref{sec:two-matching} proves Theorem~\ref{thm:two-matching}. Section~\ref{sec:half-delta} establishes Theorem~\ref{thm:half-delta}. Section~\ref{sec:f-k-delta} proves Theorem~\ref{thm:half-n-min} and Theorem~\ref{thm:f-k-delta}. Section~\ref{sec:conclusion} suggests several problems.

\section{Preliminaries}\label{sec:prelim}
This section introduces basic notation and several classical results used throughout the paper.

Let $G$ be a graph with vertex set $V(G)$ and edge set $E(G)$. For each $v\in V(G)$, the {\it degree} of $v$ is $d_G(v)$, the number of edges incident to $v$. For $k\ge 1$, define
$V_k(G)=\{v\in V(G): d_G(v)=k\}$ and $V_{\ge k}(G)=\{v\in V(G): d_G(v)\ge k\}$.

The {\it open neighborhood} of $v$ is $N_G(v)=\{u\in V(G): uv\in E(G)\}$, and the {\it closed neighborhood} is $N_G[v]=N_G(v)\cup\{v\}$. For $X\subseteq V(G)$, let
$N_G(X)=\left(\bigcup_{x\in X} N_G(x)\right)\setminus X$.
If $H$ is a subgraph of $G$ and $v\in V(G)\setminus V(H)$, define $d_H(v)=|N_G(v)\cap V(H)|$. For $M\subseteq E(G)$, let $V(M)$ denote the set of vertices incident to at least one edge of $M$.

For disjoint sets $A,B\subseteq V(G)$, let $E_G[A,B]$ be the set of edges with one endpoint in $A$ and the other in $B$, and let $G[A,B]$ be the subgraph with vertex set $A\cup B$ and edge set $E_G[A,B]$. When $A=\{x\}$, we write $E_G[x,B]$ for $E_G[A,B]$.

We will repeatedly use the following results.

\begin{theorem}[Diestel~\cite{Diestel2017}, p.~30]\label{Dirac}
Every connected graph $G$ contains a path of length at least $\min\{2\delta(G),\,|V(G)|-1\}$.
\end{theorem}

\begin{theorem}[Dirac~\cite{Dirac1967}]\label{Size:Dirac}
A minimally $2$-connected graph of order $n\ge 4$ has size at most $2n-4$.
\end{theorem}

\begin{theorem}[Halin~\cite{Halin1971}]\label{Minimumdegree: Halin}
If $G$ is minimally $k$-connected, then $\delta(G)=k$.
\end{theorem}

\begin{theorem}[Mader~\cite{Mader1972A}]\label{Forest: Mader}
Let $G$ be minimally $k$-connected and let $T$ be the set of vertices of degree $k$. Then $G-T$ is a (possibly empty) forest.
\end{theorem}

\begin{theorem}[Mader~\cite{Mader1972A}]\label{Count: Mader}
A minimally $k$-connected graph of order $n$ has at least $\frac{(k-1)n+2k}{2k-1}$ vertices of degree $k$.
\end{theorem}

\begin{theorem}[Mader~\cite{Mader1972A}]\label{Size: Mader}
Let $G$ be a minimally $k$-connected graph of order $n$. If $n\ge 3k-2$, then $|E(G)|\le k(n-k)$. Moreover, if $n\ge 3k-1$, equality holds if and only if $G\cong K_{k,n-k}$.
\end{theorem}

\begin{theorem}[Hall~\cite{Hall1935}]\label{Matching: Hall}
A bipartite graph $G$ with parts $X$ and $Y$ has a matching that covers $X$ if and only if $|N_G(S)|\ge |S|$ for every $S\subseteq X$.
\end{theorem}


\section{$k$-removable $2$-matchings}\label{sec:two-matching}

In this section, we study $k$-removable $2$-matchings in $k$-connected graphs.

\begin{lemma} \label{1-R-1}
Every connected graph $G$ contains a $1$-removable $1$-matching, unless $G$ is a tree.
\end{lemma}
\begin{pf}
If $G$ is not a tree, then $G$ contains a cycle $C$. Choose $e\in E(C)$. Then $\{e\}$ is our desired $1$-removable $1$-matching.
\end{pf}\hfill$\qed$

\begin{lemma}\label{1-R-2}
Every connected graph $G$ with $\delta(G)\ge 2$ contains a $1$-removable $2$-matching, unless $G$ is a cycle.
\end{lemma}
\begin{pf}
Let $G$ be a connected graph with $\delta(G)\ge 2$ such that removing any $2$-matching $M$ disconnects $G$. We first show that $G$ has no vertex-disjoint cycles. Suppose otherwise, and let $C_1$ and $C_2$ be vertex-disjoint cycles in $G$. For each $i\in\{1,2\}$, choose an edge $e_i\in E(C_i)$. Since $G-e_1-e_2$ remains connected, $\{e_1,e_2\}$ forms a $1$-removable $2$-matching, a contradiction.

If $G$ has a bridge $e$, then each component of $G-e$ must contain a cycle since $\delta(G)\ge 2$, contradicting the absence of vertex-disjoint cycles. Hence, $G$ is $2$-edge-connected. If $G$ is not a single cycle, it contains two cycles $C_1$ and $C_2$ sharing at least one vertex. Select an edge $e_1\in E(C_1)\setminus E(C_2)$. If $C_2$ has an edge $e_2$ not adjacent with $e_1$, then $\{e_1,e_2\}$ forms the desired matching. Otherwise, $K_4^-\subseteq C_1\cup C_2$, where $K_4^-$ denotes the graph obtained from $K_4$ by deleting one edge. In this case, there exists a $2$-matching $M\subseteq E(C_1\cup C_2)$ such that $G-M$ remains connected.
\end{pf}\hfill$\qed$

\begin{lemma}\label{2-R-2} Every $2$-connected graph $G$ with $\delta(G)\ge 3$ contains a $2$-removable $2$-matching.
\end{lemma}
\begin{pf}
For a contradiction, suppose that $G-M$ is not $2$-connected for any $2$-matching $M$ in $G$. By Theorem \ref{thm:ckl}, there exists a vertex $v\in V(G)$ such that $G-v$ is $2$-connected. Let $H=G-v$. If $H$ contains an edge $e\in E(H)$ such that $H-e$ is $2$-connected, then there exists an edge $e'\in E(G)$ such that $v\in V(e')$ and $V(e)\cap V(e')=\emptyset$ because $\delta(G)\ge 3$, meaning that $\{e,e'\}$ forms our desired matching. Hence, we may assume that $H$ is a minimally $2$-connected graph. Let $X\subseteq V(H)$ be the set of vertices of degree $2$ in $H$. Since $\delta(G)\ge 3$, this forces $X\subseteq N_G(v)$ and $X\neq\emptyset$.

\begin{claim}\label{noedge}
The subgraph $G[N_G(v)]$ has no edges.
\end{claim}
Suppose that $G[N_G(v)]$ contains an edge $e=xy$. Then we obtain a $2$-connected spanning subgraph of $G$ by adding the edges $vx$ and $vy$ to $H-xy$ and $v$. Since $\delta(G)\ge 3$, there exists a vertex $z\in  V(H)$ such that $vz\in E(G)$ and $z\not\in\{x, y\}$. Consequently, $\{xy, vz\}$ is our desired matching in $G$, a contradiction. This completes the proof of Claim \ref{noedge}.

If $H$ is a cycle, then $G$ is the join of $H$ and $v$, which implies $G[N_G(v)]$ is a cycle,  contradicting Claim \ref{noedge}. Thus $H$ is not a cycle.

Let $F$ denote the last ear in an ear decomposition of $H$. Since $H$ is minimally $2$-connected, $F$ has at least three vertices, and both ends of $F$ have degree at least three in $H$. Since $X\subseteq N_G(v)$, by Claim~\ref{noedge}, $F$ has exactly three vertices. Write $F=xuy$.
Since $\delta(G)\ge 3$,  by Claim~\ref{noedge}, there exists $w,z\in N_G(v)$ such that $\{w,z\}\cap \{u,x,y\}=\emptyset$. This implies that we can construct a $2$-connected spanning subgraph of $G$ by adding the edges $vz$ and $vu$ to $H-ux$ and $v$. Consequently, $\{vw, ux\}$ is our desired matching in $G$.
This is a contradiction.
\end{pf}\hfill$\qed$

\begin{lemma}\label{k-R-2} For $k\ge 2$, every $k$-connected graph $G$ with $\delta(G)\ge k+1$ contains a $k$-removable $2$-matching.
\end{lemma}
\begin{pf} We prove the lemma by induction on $k$. The case $k=2$  follows from Lemma \ref{2-R-2}. For $k\ge 3$, by contradiction, suppose that there exists no such matching $M$ in a $k$-connected graph $G$ of order $n$ with minimum degree at least $k+1$. Let $F$ be a maximal edge set such that $G-F$ is still $k$-connected. By the maximality of $F$, $G-F$ is minimally $k$-connected.
If $F$ contains two independent edges, then these two edges form our desired matching. Thus, $G[F]$ is a star or a triangle.

Suppose there is a triangle $T$ such that $G-E(T)$ is minimally $k$-connected. Since $\delta(G)\ge k+1$, each vertex with degree $k$ in $G-E(T)$ must belong to $V(T)$. By Theorem~\ref{Count: Mader}, $3 = |V(T)| \ge \frac{(k-1)n + 2k}{2k-1}$, which gives $n \le 4$. Thus, $n = 4$ and $k = 2$, so $G$ is $K_4$. However, this contradicts the existence of a triangle $T$ in $G$ such that $G - E(T)$ remains $2$-connected.

Thus, we can assume that $G[F]$ is a star for each maximal removable edge set $F$.  Choose $F$ to be minimal with this property. Let $v$ be the center of the star $G[F]$, and let $S=V(F)\setminus\{v\}$. Then $F=E_G[v,S]$, since $G[F]$ is a star with center $v$.

Let $Z=V(G)\setminus N_G[v]$. We prove that $G[Z]$ is a forest. Since $G-F$ is minimally $k$-connected and every vertex with degree $k$ in $G-F$ lies in $N_G[v]$, Theorem \ref{Forest: Mader} implies that $G-N_G[v]=G-F-N_G[v]$ is a forest.

Since $G-v$ is a $(k-1)$-connected graph with $\delta(G-v)\ge (k-1)+1$, the induction hypothesis yields edges $e_1, e_2\in E(G-v)$ such that $G-v-e_1-e_2$ is $(k-1)$-connected. We will show that $e_1$ and $e_2$ are the desired edges, thereby obtaining a contradiction. To see this, suppose that $G-e_1-e_2$ is not $k$-connected. Then $G-e_1-e_2$ has a $(k-1)$-separating set $Q$. If $v\in Q$, then $Q-v$ forms a $(k-2)$-separating set of $G-v-e_1-e_2$, a contradiction. Thus, $v\not\in Q$ and this implies $Z\neq\emptyset$ and there exists a component $C$ of $G-e_1-e_2-Q$ such that $v\not\in V(C)$. For the component $C$, the following claim holds.

\begin{claim}\label{componentC}
(1) $V(C)\subseteq Z$.

(2) $C$ is a path with at least two vertices, write $C=u_1u_2\ldots u_{\ell}$.

(3) $u_1, u_{\ell}\in V(e_1)\cup V(e_2)$; moreover, if $u_1\in V(e_1)$, then $u_{\ell}\in V(e_2)$.

(4) $d_G(u_i)=k+1$ for all $1\le i\le \ell$.

(5) $N_G(u_i)\cap (V(G)\setminus V(C))=Q$ for $2\le i\le \ell-1$, and $N_G(u_i)\cap (V(G)\setminus V(C))\subseteq Q\cup V(e_i)$ for $i\in\{1, \ell\}$.
\end{claim}

Since $e_1, e_2\in E(G-v)$, the vertex set $N_G[v]\setminus Q$ belongs to a component of $G-e_1-e_2-Q$. As $v\notin V(C)$, we have that $V(C)\subseteq Z$. Thus (1) holds.

Since $G[Z]$ is a forest, $G[V(C)]$ is a tree by (1). Thus $e_1,e_2\not\in E(G[V(C)])$. If $C$ is not a path, then there exists a leaf $u$ such that $u\notin V(e_1)\cup V(e_2)$. Thus, $d_G(u)\le |Q|+1=k-1+1=k$, a contradiction. So $C$ is a path. If $C$ is trivial, then its unique vertex has degree at most $k$, a contradiction. Therefore, $C$ is a path with at least two vertices. Thus, (2) holds. Considering the degree of $u_i$, we obtain (3)-(5) hold. Thus, claim \ref{componentC} holds.

By Theorem \ref{Count: Mader}, $|V(F)|\ge\frac{(k-1)n+2k}{2k-1}>2$ since $n\ge \delta(G)+1\ge k+2\ge 5$. Thus, we have that $|S|\ge 2$.

Together with the minimality of $F$, this implies that no vertex of $C$ can be chosen as $v$. Hence, for any $e\in E(C)$, $G-e$ is not $k$-connected. Therefore, $G-u_1u_2$ has a $(k-1)$-separating set $U$. By the construction of $C$, this forces $U=Q$. For each $i\in\{1,2\}$, let $C_i$ be a component of $G-u_1u_2-U$ containing $u_i$. Since $\delta(G)\ge k+1$, it follows that $V(C_i)\setminus\{u_i\}\neq\emptyset$ for each $i\in\{1, 2\}$. Since $v\notin U$, there exists $i\in \{1, 2\}$ such that all vertices of $C_i$ have degree at least $k+1$ in $G-F$, and by Theorem \ref{Forest: Mader}, $C_i$ is a tree. Since $|U|=|Q|\le k-1$, we can find a leaf $z$ in $C_i$ such that $d_G(z)<k+1$, a contradiction.
\end{pf}\hfill$\qed$

Combining Lemmas \ref{1-R-2} and \ref{k-R-2}, we obtain Theorem \ref{thm:two-matching}.

\section{$k$-removable $\lceil\frac{\delta+1}{2}\rceil$-matchings for $k\in\{1,2,3\}$}\label{sec:half-delta}

In this section, we further investigate larger $k$-removable matchings in $k$-connected graphs with minimum degree at least $k+1$, for $k\in\{1,2,3\}$.

\begin{lemma}\label{2-R-3}
Every $2$-connected graph $G$ with $\delta(G)=4$ contains a $2$-removable $3$-matching, unless $G\cong K_5$.
\end{lemma}
\begin{pf}
For a contradiction, suppose that $G\not\cong K_5$ and $G-M$ is not $2$-connected for any $3$-matching $M$ in $G$. By Theorem \ref{thm:ckl} and since $\delta(G)\ge 4$, there exist two vertices $x$ and $y$ such that $G-\{x, y\}$ is $2$-connected. Let $H=G-\{x, y\}$ and let $h=|V(H)|$.

\noindent{\it Case~1.} $H$ is not a minimally $2$-connected graph.

Let $F$ be an edge set in $H$ such that $H-F$ is minimally $2$-connected.
If $|(N_G(x)\cup N_G(y))\cap V(H)|\ge 4$, then for any edge $e\in F$, there exist $e_1\in E_G[x,V(H)]$ and $e_2\in E_G[y,V(H)]$ such that $\{e,e_1,e_2\}$ is a $2$-removable $3$-matching in $G$, as $|N_G(x)\cap V(H)|\ge 3$ and $|N_G(y)\cap V(H)|\ge 3$. Hence, $|(N_G(x)\cup N_G(y))\cap V(H)|\le 3$. Since $d_G(x)\ge \delta(G)=4$, we have that $|(N_G(x)\cup N_G(y))\cap V(H)|=3$.

If there is an edge $e\in F$  with $V(e)\not\subseteq (N_G(x)\cup N_G(y))\cap V(H)$, then there exist $e_1\in E_G[x,V(H)]$ and $e_2\in E_G[y,V(H)]$ such that $\{e,e_1,e_2\}$ is a $2$-removable $3$-matching  in $G$. Thus, $V(e)\subseteq (N_G(x)\cup N_G(y))\cap V(H)$ for each $e\in F$. Since $\delta(G)\ge 4$, It follows that $H-F$ has at most three vertices of degree $2$.
On the one hand, $|E(H-F)|\ge \frac{3\times 2}{2} +\frac{(h-3)\times4}{2}=2h-3$. On the other hand, by Theorem \ref{Size:Dirac}, $|E(H-F)|\le 2h-4$, a contradiction.

\noindent{\it Case~2.} $H$ is a minimally $2$-connected graph.

Since $G\not\cong K_5$ and $\delta(G)\ge 4$, we have that $h\ge 4$ and $V_2(H)\subseteq N_G(x)\cap N_G(y)$. By Theorem \ref{Count: Mader}, $|V_2(H)|\ge \frac{(2-1)h+2\times 2}{2\times 2-1}\ge \frac{8}{3}$.

If $|V_2(H)|=3$ and $|(N_G(x)\cup N_G(y))\cap V(H)|=3$, then $|E(H)|\ge \frac{3\times 2}{2} +\frac{(h-3)\times4}{2}=2h-3$. On the other hand, by Theorem \ref{Size:Dirac}, $|E(H)|\le 2h-4$, a contradiction.

If $|V_2(H)|\ge 4$ or $|V_2(H)|=3$ and $|(N_G(x)\cup N_G(y))\cap V(H)|\ge 4$, then $V_2(H)$ is an independent set in $H$. Otherwise, for each edge $e\in E(H[V_2(H)])$, there exist $e_1\in E_G[x,V(H)]$ and $e_2\in E_G[y,V(H)]$ such that $\{e,e_1,e_2\}$ is a $2$-removable $3$-matching in $G$. Let $F$ be the last ear in an ear decomposition of $H$. Since $H$ is minimally $2$-connected, $F$ has at least three vertices and both ends of $F$ have degree at least three in $H$. As $V_2(H)$ is an independent set in $H$, $F$ has exactly three vertices. Write $F=uvw$. This implies that the graph obtained from $H-uv$ by adding vertex and edges $xv, xv'$ is a $2$-connected graph, where $v'\in N_G(x)\setminus \{v\}$. Choose $x'\in N_G(x), y'\in N_G(y)$ such that $x',y',u$ and $v$ are distinct. Then $\{uv,xx',yy'\}$ is a $2$-removable $3$-matching in $G$.
\end{pf}\hfill$\qed$

\begin{theorem}\label{2-R-half-m}
For $k\le 2$,  every $k$-connected graph $G$ with $\delta(G)\ge k+1$ contains a $k$-removable $\lceil(\delta(G)+1)/2\rceil$-matching, unless $\delta(G)=2$ and $G$ is a cycle, or $\delta(G)\ge 4$ is even and $G$ is a complete graph of order $\delta(G)+1$.
\end{theorem}
\begin{pf}
For $k=1$, if $\delta(G) \in \{2,3\}$, Lemma \ref{1-R-2} implies $G$ contains our desired matching unless $\delta(G)=2$ and $G$ is a cycle.  For $k=2$, if $\delta(G)\in\{3, 4\}$, Lemmas \ref{2-R-2} and \ref{2-R-3} imply $G$ contains our desired matching unless $\delta(G)=4$ and $G\cong K_5$. Thus, in what follows, we assume that $\delta(G)\ge k+3$.

By repeatedly applying Theorem \ref{thm:ckl}, we can select a set $X $ of $\lceil(\delta(G)+1)/2\rceil$ vertices such that $G-X$ remains $k$-connected. This is possible because $\delta(G) - (|X|-1) = \lfloor(\delta(G)+1)/2\rfloor \ge \lfloor(k+3+1)/2\rfloor\ge k+1$ for $k\le 2$. Let $H =G-X$ and $G' = G[X,V(H)]$. For any $x \in X$, its degree in $G'$ is $d_{G'}(x) \ge d_G(x) - (|X|-1) \ge \delta(G) - (\lceil(\delta(G)+1)/2\rceil-1) \ge \lfloor(\delta(G)+1)/2\rfloor \ge k+1$ for $k\le 2$.

If $G'$ contains a matching $M$ covering $X$, then $G-M$ is $k$-connected because $H$ is $k$-connected and
$|N_G(x)\cap V(H)|=d_{G'}(x) \ge k+1$ for every $x \in X$. Thus $M$ is a $k$-removable $\lceil(\delta(G)+1)/2\rceil$-matching. Otherwise, by Hall's condition (Theorem \ref{Matching: Hall}), there exists $S \subseteq X$ with $|N_{G'}(S)| < |S|$. Since $|X|=\lceil(\delta(G)+1)/2\rceil$ and $d_{G'}(x)\ge\lfloor(\delta(G)+1)/2\rfloor$ for every $x \in X$, we have that $S=X$, $|N_{G'}(X)| < |X|$, and $d_{G'}(x)=\lfloor(\delta(G)+1)/2\rfloor$ for every $x \in X$. Since $\lceil(\delta(G)+1)/2\rceil=|X| > |N_{G'}(X)| \ge d_{G'}(x) = \lfloor(\delta(G)+1)/2\rfloor$ for every $x\in X$, it follows that $\delta(G)$ must be even, $|N_{G'}(X)|=\lfloor(\delta(G)+1)/2\rfloor$,  $G[X]$ is a complete graph of order $\lceil(\delta(G)+1)/2\rceil$, and $G'$ is a complete bipartite graph.

Since $\delta(G)$ is even, and by our assumption, $G$ is not a complete graph of order $\delta(G)+1$, it follows that $V(H) \setminus N_{G'}(X) \neq \emptyset$. Since $N_G(X)=N_{G'}(X)$, we have that $V(H) \setminus N_G(X) \neq \emptyset$. Consider the following two cases.

\noindent{\it Case~1.} $k=1$.

For a vertex $y \in V(H) \setminus N_G(X)$, if $H$ contains a cycle $C$ through $y$, let $y'$ be a neighbor of $y$ on $C$. Choose distinct vertices $x,x' \in X$. It follows that $G'-\{x,x',y,y'\}$ contains a matching $M'$ covering $X \setminus \{x,x'\}$, and $M' \cup \{xx',yy'\}$ is the desired matching. Therefore, we may assume that for any $y' \in N_{H}(y)$, $yy'$ is a cut edge in $G$. This implies that $H-y$ has exactly $d_G(y)$ components. Since $d_G(y) > |N_G(X)|$, by the Pigeonhole Principle, at least one component of $H-y$, say $H_1$, is vertex-disjoint from $N_G(X)$. By Lemma \ref{1-R-2}, the subgraph $H_1$ has a $1$-removable matching $M_1$ because $\delta(H_1) \ge 3$. Furthermore, a matching $M'$ of size $\lfloor(\delta(G)+1)/2\rfloor$ can be found in $G[X, N_G(X)]$. Then $M' \cup M_1$ is the desired matching.

\noindent{\it Case~2.} $k=2$.

Suppose there exists an edge $e$ such that $V(e)\cap (V(H)\setminus N_G(X))\neq\emptyset$ and $H-e$ is $2$-connected. Let $e=yy'$ and choose $x,x'\in X$. Then $ G'-\{x,x',y,y'\} $ admits a matching $M'$ covering $X\setminus\{x,x'\}$, and $M'\cup\{xx',yy'\}$ forms the desired matching. Hence, for any edge $e$ with $V(e)\cap (V(H)\setminus N_G(X))\neq\emptyset$, $H-e$ is not $2$-connected.

Without loss of generality, assume that $H$ is minimally 2-connected; otherwise, remove some edge set $F$ so that $H-F$ becomes minimally $2$-connected. Since $V(F)\cap (V(H)\setminus N_G(X))=\emptyset$, we have $d_{H-F}(v)=d_G(v)\ge \delta(G)\ge 5 $ for all $v\in V(H)\setminus N_G(X) $. By Theorem \ref{Forest: Mader} and because $V(H)\setminus N_G(X)\neq \emptyset$, the induced subgraph $H[V(H)\setminus N_G(X)] $ is a nonempty forest. Let $v$ be a leaf of this forest. As $d_G(v)\ge 5$, it follows that $|N_G(v)\cap N_G(X)|\ge 4$. Choose an edge $vv'$ with $v'\in N_G(X)$. Then every vertex $u\in V(H)\setminus N_G(X) $ has two internally disjoint paths to $ N_G(X) $ in $ H-F-vv' $. For two distinct $x,x'\in X$, $G'-\{x,x',v,v'\}$ contains a matching $M'$ covering $X\setminus\{x,x'\}$, and $M'\cup\{xx',vv'\}$ forms the required matching.
\end{pf}\hfill$\qed$

Before examining $3$-removable $\lceil(\delta(G)+1)/2\rceil$-matchings in $3$-connected graphs {\color{red} $G$} with $\delta(G)\ge 5$, we first establish the following lemma.

\begin{lemma}\label{3-R-4}
Every $3$-connected graph $G$ with $\delta(G)=6$ contains a $3$-removable $4$-matching, unless $G\cong K_7$.
\end{lemma}
\begin{pf}
For contradiction, suppose that $G\not\cong K_7$ and that $G-M$ is not $3$-connected for any matching $M$ of size $4$ in $G$. By Theorem~\ref{thm:ckl} and  $\delta(G)\ge 6$, there exists a vertex set $\{x, y, z\}$ such that $G-\{x, y, z\}$ is $3$-connected. Let $X=\{x, y, z\}$, $H=G-\{x, y, z\}$, and $h=|V(H)|$.

\begin{claim}\label{cla: atleast8}
If $|N_G(X)\cap V(H)|=4$, then $h\ge 8$.
\end{claim}
Since $G\cong K_7$, the set $V(H)\setminus N_G(X)$ is nonempty. Let $v\in V(H)\setminus N_G(X)$. As $d_G(v)\ge \delta(G)\ge 6$, $H$ has order at least $7$. If $h=7$, then $G[N_G(X)\cap V(H), V(G)\setminus(N_G(X)\cap V(H))\cong K_{4,6}$, implying that $G$ contains a $3$-removable $4$-matching, a contradiction. This completes the proof of Claim \ref{cla: atleast8}.

\noindent{\it Case~1.} $H$ is not a minimally $3$-connected graph.
	
Let $F\subseteq E(H)$ be an edge set such that $H-F$ is minimally $3$-connected. Since $\delta(G)\ge 6$, we have that $|N_G(v)\cap V(H)|\ge 4$ for every $v\in\{x,y,z\}$.
If $|N_G(X)\cap V(H)|\ge 5$, then for any edge $e\in F$, there exist $e_1\in E_G[x,V(H)]$, $e_2\in E_G[y,V(H)]$ and $e_3\in E_G[z,V(H)]$ such that the set $\{e,e_1,e_2,e_3\}$ forms a $3$-removable $4$-matching in $G$, a contradiction. Hence $|N_G(X)\cap V(H)|=4$. By Claim \ref{cla: atleast8}, $h\ge 8$.
	
If there exists $e\in F$ with $V(e)\not\subseteq N_G(X)\cap V(H)$, then there exist $e_1\in E_G[x,V(H)]$, $e_2\in E_G[y,V(H)]$ and $e_3\in E_G[z,V(H)]$ such that the set $\{e,e_1,e_2,e_3\}$ forms a $3$-removable $4$-matching in $G$, leading to a contradiction. Thus, $V(e)\subseteq N_G(X)\cap V(H)$ for every $e\in F$. On the one hand, since $\delta(G)\ge 6$ and $|N_G(X)\cap V(H)|=4$, $H-F$ has at most four vertices of degree $3$. On the other hand, since $h\ge 8$, by Theorem~\ref{Count: Mader}, $H-F$ has at least $\frac{(3-1)h+2\times3}{2\times3-1}>4$ vertices of degree $3$, a contradiction.
	
\noindent{\it Case~2.} $H$ is a minimally $3$-connected graph.

Suppose that $N_G(v)\cap V(H)$ is not an independent set for some $v\in \{x,y,z\}$. If $|N_G(X)\cap V(H)|\ge 5$, then for any edge $e$ such that $V(e)\subseteq N_G(v)\cap V(H)$, there exist edges $e_1\in E_G[x,V(H)]$, $e_2\in E_G[y,V(H)]$, and $e_3\in E_G[z,V(H)]$ such that the set $\{e,e_1,e_2,e_3\}$ forms a $3$-removable $4$-matching in $G$, a contradiction. Hence $|N_G(X)\cap V(H)|= 4$. Since $\delta(G)=6$, we have that $H$ has at most four vertices of degree $3$. By Claim~\ref{cla: atleast8}, $h\ge 8$. By Theorem~\ref{Count: Mader}, $H$ has at least $4$ vertices of degree $3$, a contradiction.

Therefore, $N_G(v)\cap V(H)$ is an independent set for every $v\in X$. Let $Y=\{v\in V(H): \{x,y,z\}\subseteq N_G(v)\}$, and let $H'=G[X,Y]$. Since $G\not\cong K_7$ and $\delta(G)\ge 6$, we conclude that $h\ge 5$ and $V_3(H)\subseteq Y$. By Theorem \ref{Count: Mader}, we have that $|V_3(H)|\ge \frac{(3-1)h+2\times3}{2\times 3-1}>3$. Thus, $|Y|\ge |V_3(H)|\ge 4$ and $H'$ is $3$-connected.

If $|V(H)\setminus Y|=3$, then $G[Y,V(G)\setminus Y]$ is a complete bipartite graph with $|V(G)\setminus Y|=6$ and $|Y|\ge 4$. Thus, $G$ contains a $3$-removable $4$-matching, a contradiction.

Suppose $|V(H)\setminus Y|\ge 4$. Since $H$ is minimally $3$-connected, Lemma \ref{Forest: Mader} implies that $H-Y$ is a forest. Since $H$ is $3$-connected, for any $v\in V(H)\setminus  Y$, there exist three internally disjoint paths from $v$ to $Y$ in $H$.

Let $Z = \{v_1, \dots, v_{\ell}\}$ be the set of leaves in $H-Y$, where $\ell \ge 2$. If $\ell\ge 4$, then since $|N_G(v_i)\cap Y|\ge \delta(G)-1\ge 5$ for $i\in \{1,2,3,4\}$, we can choose $v_i'\in N_G(v_i)\cap Y$ such that $M=\{v_1v_1',v_2v_2',v_3v_3',v_4v_4'\}$ is a $4$-matching. Since there exist three internally disjoint paths from $v$ to $Y$ for any $v\in V(H)\setminus Y$ in $H-M$, the matching $M$ is a $3$-removable $4$-matching in $G$, a contradiction. If $\ell= 3$, then choose $u_1\in V(H)\setminus Y$ such that $v_1u_1\in E(H-Y)$. Since $|N_G(v_i)\cap Y|\ge 5$ for $i\in \{1,2,3\}$ and $|N_G(u_1)\cap Y|\ge 3$, we can choose $v_i'\in N_G(v_i)\cap Y$ and $u_1'\in N_G(u_1)\cap Y$ such that $\{v_1v_1',v_2v_2',v_3v_3',u_1u_1'\}$ is a $3$-removable $4$-matching in $G$, a contradiction. If $\ell=2$, then choose $u_1,u_2\in V(H)\setminus Y$ such that $v_1u_1\in E(H-Y)$ and $v_2u_2\in E(H-Y)$. Since $|N_G(v_i)\cap Y|\ge 5$ and $|N_G(u_i)\cap Y|\ge 4$ for $i\in \{1,2\}$, we can choose $v_i'\in N_G(v_i)\cap Y$, $u_i'\in N_G(u_i)\cap Y$ such that $\{v_1v_1',v_2v_2',u_1u_1',u_2u_2'\}$ is a $3$-removable $4$-matching in $G$, a contradiction.
\end{pf}\hfill$\qed$

\begin{theorem}\label{3-R-half-m} Every $3$-connected graph $G$ with $\delta(G)\ge 5$ contains a $3$-removable $\lceil(\delta(G)+1)/2\rceil$-matching, unless $\delta(G)$ is even and $G\cong K_{\delta(G)+1}$.
\end{theorem}
\begin{pf}
We first consider the case $\delta(G)=5$. By Theorem \ref{thm:ckl}, there exists a vertex $v\in V(G)$ such that $G-v$ is $3$-connected. Since $\delta(G-v)\ge 4$, applying Theorem \ref{thm:two-matching} to $G-v$, we obtain a $3$-removable matching $M$ of size $2$ in $G-v$. Since $\delta(G)\ge 5$, there exists an edge $e\in E(G)$ such that $v\in V(e)$ and $V(e)\cap V(M)=\emptyset$. Consequently, the set $M\cup\{e\}$ is a $3$-removable $3$-matching  in $G$. For $\delta(G)=6$, the result follows from Lemma \ref{3-R-4}.

Next, assume that $\delta(G)\ge 7$ and that $G\not\cong K_{\delta(G)+1}$ for even $\delta(G)$. Since $\delta(G)-(\lceil(\delta(G)+1)/2\rceil-1)=\lfloor(\delta(G)+1)/2\rfloor\ge 4$,  by Theorem \ref{thm:ckl}, $G$ contains a vertex set $X$ of
size $\lceil(\delta(G)+1)/2\rceil$ such that $G-X$ is $3$-connected. Define $H=G-X$ and $G'=G[X,V(H)]$.
If $G'$ contains a matching $M$ covering $X$, then $G-M$ is $3$-connected since $H$ is $3$-connected and $|N_G(x)\cap V(H)|\ge\delta(G)-(\lceil(\delta(G)+1)/2\rceil-1)=\lfloor(\delta(G)+1)/2\rfloor \ge 4$ for each $x\in X$. Hence $M$ is a $3$-removable $\lceil(\delta(G)+1)/2\rceil$-matching, and we are done. Thus we may assume that $G'$ has no matching covering $X$.

If $\delta(G)$ is odd, then $G'$ contains a matching covering $X$, because for each $x \in X$, $d_{G'}(x)=|N_G(x) \cap V(H)| \ge \delta(G) - (\lceil(\delta(G)+1)/2\rceil - 1) \ge \lfloor(\delta(G)+1)/2\rfloor=\lceil(\delta(G)+1)/2\rceil$. This contradicts the assumption that $G'$ has no matching covering $X$. We may therefore assume that $\delta(G)$ is even. By the theorem's assumption, $G$ is not a complete graph of order $\delta(G)+1$. Thus $V(H) \setminus N_{G'}(X) \neq \emptyset$. As $N_G(X)=N_{G'}(X)$, we also have that $V(H) \setminus N_G(X) \neq \emptyset$.

If there exists a vertex $x\in X$ such that $d_{G'}(x)\ge \delta(G)-(\lceil(\delta(G)+1)/2\rceil-1)+1\ge \lfloor(\delta(G)+1)/2\rfloor +1= \lceil(\delta(G)+1)/2\rceil$, then $G'$ contains a matching covering $X$, which also leads to a contradiction. Thus, $d_{G'}(x)=\lfloor(\delta(G)+1)/2\rfloor$ for each $x\in X$. If $N_{G'}(x)\neq N_{G'}(x')$ for some distinct $x,x'\in X$, then $G'$ contains a matching covering $X$, a contradiction. Consequently, $N_{G'}(x)=N_{G'}(x')$ for any distinct $x,x'\in X$. Therefore, $G[X]$ is a complete graph of order $\lceil(\delta(G)+1)/2\rceil$, and $G[X, N_G(X)]$ is a complete bipartite graph.

If there exists an edge $e\in E(H)$ such that $V(e)\cap (V(H)\setminus N_G(X))\neq\emptyset$ and $H-e$ is $3$-connected, then $G[X\setminus \{x_1,x_2\},N_G(X) \setminus V(e)]$ contains a matching $M$ of size $\lceil(\delta(G)+1)/2\rceil-2$ such that $G-(M\cup \{e,x_1x_2\})$ is $3$-connected, where $x_1,x_2\in X$, a contradiction. Thus, for each edge $e\in E(H)$ such that $H-e$ is $3$-connected, we have that $V(e)\subseteq N_G(X)$. Let $F\subseteq E(H)$ be such that $H-F$ is minimally $3$-connected. Then $V(F)\subseteq N_G(X)$.

Since $d_H(u)=d_G(u)\ge \delta(G)\ge 7$ for every $u\in V(H)\setminus N_G(X)$, we have that $V_3(H-F)\subseteq N_G(X)$. By Theorem \ref{Forest: Mader}, $H-N_G(X)$ is a forest. Since $V(H) \setminus N_G(X) \neq \emptyset$, we can choose a leaf $u$ in $H-N_G(X)$. Then $d_G(u)=d_H(u)\le 1+|N_G(X)|=1+\lfloor(\delta(G)+1)/2\rfloor<\delta(G)$ for $\delta(G)\ge 7$, a contradiction.
\end{pf}\hfill$\qed$

By combining Theorems \ref{2-R-half-m} and \ref{3-R-half-m}, we obtain Theorem \ref{thm:half-delta}.

For $k \ge 4$, if the minimum degree is increased to $3k - 3$, then we can find a $k$-removable $\lceil(\delta(G)+1)/2\rceil$-matching, except when $\delta(G)$ is even and $G \cong K_{\delta(G)+1}$. This establishes Theorem~\ref{thm:3k-1}. The proof proceeds as follows.

\noindent{\bf The Proof of Theorem~\ref{thm:3k-1}.}
By repeatedly applying Theorem \ref{thm:ckl}, we can select a set $X$ of $\lceil(\delta(G)+1)/2\rceil$ vertices such that $G-X$ remains $k$-connected. This is possible because  $\delta(G) - (|X|-1) = \lfloor(\delta(G)+1)/2\rfloor \ge \lfloor 3k/2\rfloor$ for $\delta(G) \ge 3k-1$. Let $H=G-X$ and $G' = G[X,V(H)]$. For any $x \in X$, we have that $d_{G'}(x) \ge d_G(x) - (|X|-1) \ge \delta(G) - (\lceil(\delta(G)+1)/2\rceil-1) \ge  \lfloor(\delta(G)+1)/2\rfloor\ge \lfloor 3k/2\rfloor$ for $\delta(G) \ge 3k-1$. Since $k \ge 4$, we have that $\delta(G) \ge 9$.

If $G'$ contains a matching $M$ covering $X$, then $M$ is our desired matching because $H$ is $k$-connected and  $d_{G'}(x) \ge  \lfloor 3k/2\rfloor\ge k+1$ for all $x \in X$ when $k\ge 4$. Otherwise, by Hall's Theorem (Theorem \ref{Matching: Hall}), there exists $S \subseteq X$ with $|N_{G'}(S)| < |S|$. Given $|X|=\lceil(\delta(G)+1)/2\rceil$ and $d_{G'}(x)\ge\lfloor(\delta(G)+1)/2\rfloor$ for all $x \in X$, this implies $S=X$, $|N_{G'}(X)| < |X|$, and $d_{G'}(x)=\lfloor(\delta(G)+1)/2\rfloor$ for all $x \in X$. Since $\lceil(\delta(G)+1)/2\rceil=|X| > |N_{G'}(X)| \ge d_{G'}(x) = \lfloor(\delta(G)+1)/2\rfloor$, it follows that $\delta(G)$ must be even, $|N_{G'}(X)|=\lfloor(\delta(G)+1)/2\rfloor$,  $G[X]$ is a complete graph of order $\lceil(\delta(G)+1)/2\rceil$, and $G[X, N_{G'}(X)]$ is a complete bipartite graph.

Since $\delta(G)$ is even, by the theorem's assumption, $G$ is not a complete graph of order $\delta(G)+1$. Hence, $V(H) \setminus N_{G'}(X) \neq \emptyset$. As $N_G(X)=N_{G'}(X)$, we also have that $V(H) \setminus N_G(X) \neq \emptyset$.

Let $F \subseteq E(H)$ such that $H - F$ is minimally $k$-connected. If some $e \in F$ satisfies $V(e) \not\subseteq N_G(X)$, then there exists a $(\lceil(\delta(G)+1)/2\rceil - 2)$-matching $M'$ in $G[X \setminus \{x, x'\}, N_G(X) \setminus V(e)]$ for distinct $x, x' \in X$, such that $M' \cup \{e, xx'\}$ forms a $k$-removable $\lceil (\delta(G) + 1)/2 \rceil$-matching  in $G$. Hence, $V(e) \subseteq N_G(X)$ for all $e \in F$. Since $k\ge 4$ and $d_H(u)=d_G(u)\ge \delta(G) \ge 3k - 1>k$ for every $u\in V(H)\setminus N_G(X)$, it follows that $V_k(H - F) \subseteq N_G(X)$. By Theorem~\ref{Forest: Mader}, $H - N_G(X)$ is a forest. As $V(H) \setminus N_{G'}(X) \ne \emptyset$, this forest has at least one leaf. For such a leaf $u$, we have $d_G(u)=d_{G'}(u) \le 1 + |N_G(X)| \le 1 + \lfloor (\delta(G) + 1)/2 \rfloor < \delta(G)$ for $\delta(G) \ge 11$, yielding a contradiction.
\hfill$\qed$

\section{Function $f(k,\delta)$}\label{sec:f-k-delta}
In this section, we study the function $f(k,\delta)$.
By Theorem~\ref{thm:3k-1}, $f(k,\delta)\ge(\delta+1)/2$ for $\delta\ge 3k-3$.

Combined with $f(k,\delta) \le \delta$, the following theorem establishes that $f(1,\delta) = \delta$ for $\delta \ge 3$.

\begin{theorem}\label{thm:1-R-de}
For $\delta\ge 3$, every connected graph $G$ with order $n\ge 2\delta$ and $\delta(G)\ge \delta$ contains a $1$-removable $\delta$-matching.
\end{theorem}

\begin{pf} By Theorem \ref{Dirac},  $G$ contains a path $P$ of length at least $\min\{2\delta(G), n-1\}$. Since $n\ge 2\delta$ and $\delta(G)\ge \delta$, the length of $P$ is at least $2\delta-1$; hence, $G$ contains a matching $M$ of size $\delta$. If $G-M$ is connected, then $M$ is our desired matching. Thus, we suppose that $G-M$ is disconnected. Let $M'$ be a maximal subset of $M$ such that $G-M'$ remains connected. Choose $e\in M\setminus M'$. By the maximality of $M'$, $G-(M'\cup \{e\})$ has exactly two components, denoted by $G_1$ and $G_2$. Since $G_i$ is connected and $\delta(G_i)\ge \delta(G)-1=\delta-1$, by Theorem \ref{2-R-half-m}, $G_i$ contains a $1$-removable matching $M_i$ with size $\lceil\delta/2\rceil$, unless $\delta-1=2$ and $G_i$ is a cycle, or $\delta-1\ge 4$ is even and $G_i$ is a complete graph of order $\delta$. If, for each $i\in\{1, 2\}$, $G_i$ is neither a cycle with $\delta-1=2$ nor a complete graph of order $\delta$ with even $\delta-1\ge 4$, then $M_1\cup M_2$ is the desired matching.

Suppose at least one of $G_1$ and $G_2$ is a cycle of order at least four. Without loss of generality, let $G_1$ be such a cycle. Then $\delta(G) = 3$ and $|E_G[V(G_1), V(G_2)]| \ge 4$, which contradicts the assumption that $M' \cup \{e\}$ is an edge cut of size at most 3 separating $G_1$ and $G_2$. Therefore, $\delta - 1 \ge 2$ is even, $G_i$ is a complete graph of order $\delta$, and $M' \cup \{e\} = M = E_G[V(G_1), V(G_2)]$ for $i \in \{1,2\}$. Choose $x_1y_1, x_2y_2 \in M$ with $x_i,y_i \in V(G_i)$ for $i=1,2$. Then $(M \setminus \{x_1y_1, x_2y_2\}) \cup \{x_1x_2, y_1y_2\}$ is the desired matching.
\end{pf}\hfill$\qed$

\begin{remark}
The order condition in Theorem \ref{thm:1-R-de} is optimal, since $K_{2\delta-1}$  has no $\delta$-matching. The minimum degree condition in Theorem \ref{thm:1-R-de} is also optimal, since  $K_{\delta-1,\delta+1}$ has no $\delta$-matching.
\end{remark}

With Theorem~\ref{thm:1-R-de} established, we are now prepared to prove Theorem~\ref{thm:half-n-min}.

\noindent{\bf The Proof of Theorem~\ref{thm:half-n-min}.}
If $n\ge 2\delta(G)$, then the desired result follows from Theorem \ref{thm:1-R-de}.
Suppose that $n<2\delta(G)$. By Dirac's theorem, $G$ contains a Hamiltonian cycle, and hence there exists a matching $M$ of size $\lfloor n/2\rfloor$ in $G$.
If $G-M$ is connected, then we are done. Thus, we may assume that $G-M$ is disconnected, and let $G_1$ and $G_2$ denote the two components of $G-M$. For $i=1, 2$, since $|V(G_i)|-1\ge\delta(G_i)\ge \delta(G)-1$, we have $|V(G_i)|\ge \delta(G)$.
Thus, $n\ge|V(G_1)|+|V(G_2)|\ge 2\delta(G)$, a contradiction.\hfill$\qed$

\begin{remark}
The size of $1$-removable matching in Theorem \ref{thm:half-n-min} is optimal, since $K_{\delta,\delta}$ has no $(\delta+1)$-matching.
\end{remark}

To study $f(2,\delta)$, we need the following lemma.

\begin{lemma}\label{separating set}
For $m > 2$, let $G$ be a 2-connected graph containing a unique vertex $x$ such that $d_G(x) < m$ and $d_G(z)\ge m$ for $z\in V(G)\setminus\{x\}$. Then there exists a neighbor $y$ of $x$ such that $\{x, y\}$ is not a separating set of $G$.
\end{lemma}
\begin{pf}
For contradiction, suppose that $G-\{x,y\}$ is disconnected for every $y\in N_G(x)$. Choose $y\in N_G(x)$ such that $G-\{x,y\}$ has a smallest component, say $G_1$. $G_1$ contains at least two vertices because $d_G(x_1)\ge m>2$.
Let $G_2,\dots, G_t$ be the other components of $G-\{x,y\}$ besides $G_1$. Since $G$ is $2$-connected, $N_G(x)\cap V(G_i)\ne \emptyset$ and $N_G(y)\cap V(G_i)\ne \emptyset$ for $i\in [1,t]$.

For $x_1\in N_G(x)\cap V(G_1)$, if $G-\{x,x_1\}$ is connected, then we are done. Suppose that $G-\{x,x_1\}$ is not connected.
Since $N_G(y)\cap V(G_i)\ne \emptyset$ for every $i\in \{1,\dots,t\}$, $G-\{x,x_1\}$ has a component $G'_1$ such that $G'_1\subseteq G_1$ and $|V(G'_1)|<|V(G_1)|$, a contradiction.
\end{pf}\hfill$\qed$

For a subgraph $H$ of $G$, define $G/H$ as the graph formed from $G-V(H)$ by adding a new vertex $v$ and edges $vu$ for every $u\in N_G(V(H))$.

\begin{theorem}\label{2Connected(m-3)matching}
For $\delta\ge 5$, every $2$-connected graph $G$ with $|V(G)|\ge 2(\delta-2)$ and $\delta(G)\ge \delta$ contains a $2$-removable $(\delta-3)$-matching. In particular, if $\delta$ is even, then $G$ contains a $2$-removable $(\delta-2)$-matching.
\end{theorem}
\begin{pf}
For $\delta\ge 5$, assume for contradiction that, for any matching $M$ of size $\delta-3$ (or $\delta-2$ if $\delta$ is even), $G-M$ is not $2$-connected.

By Theorem \ref{Dirac}, $G$ has a matching $M$ of size $\delta-2$. Let $M'\subseteq M$ be a minimal matching such that $G-M'$ is not $2$-connected. Let $x$ be a cut vertex of $G-M'$, so $G-M'-x$ has exactly two components, $G_1$ and $G_2$, and $M'=E_G[G_1,G_2]$.

Define $G'_1=G/(G_2\cup\{x\})$ and $G'_2=G/(G_1\cup\{x\})$, where $x_i$ denotes the new vertex in $G'_i$ for each $i\in\{1,2\}$.
One can check that for each $v\in V(G'_i)\setminus\{x_i\}$, $d_{G'_i}(v)\ge \delta-1$.

We assert that each $G'_i$ is $2$-connected for $i = 1, 2$. Assume to the contrary that some vertex $u$ is a cut vertex of $G'_i$. Because $G - M' - x$ has exactly two components, $x_i$ cannot coincide with $u$. Let $H_i$ be a component of $G'_i - u$ that does not contain $x_i$. Then the set $N_{G'_i}(x_i)\setminus\{u\}$ lies entirely in an another component of $G'_i - u$. By $N_{G'_i}(x_i)\setminus\{u\}=N_G(V(G_j)\cup\{x\})\setminus\{u\}$, it follows that $u$ is also a cut vertex of $G$, a contradiction. Hence, each $G'_i$ is $2$-connected for $i = 1, 2$.

Consider the following two cases for each $i\in\{1,2\}$:

\noindent{\it~Case 1.} $d_{G'_i}(x_i)\ge \delta-1$.

By Theorem \ref{2-R-half-m}, $G'_i$ has a $2$-removable $\lceil \delta/2 \rceil$-matching  unless $G'_i \cong K_{\delta-1}$ and $\delta-1$ is even. Thus $G'_i$ has a $2$-removable matching $M_i$ of size $\lceil \delta/2 \rceil - 1$ avoiding $x_i$.

\noindent{\it~Case 2.} $d_{G'_i}(x_i)< \delta-1$.

By Lemma \ref{separating set}, there exists $y_i \in N_{G'_i}(x_i)$ such that $G'_i-\{x_i, y_i\}$ is connected. Thus, $G'_i/\{x_i, y_i\}$ is $2$-connected with minimum degree at least $\delta-2$. Since $\delta(G'_i/x_iy_i)\ge \delta-2$, as in Case 1, $G'_i/\{x_i, y_i\}$ has a $2$-removable matching $M_i$ of size $\lceil \frac{\delta-1}{2} \rceil-1$ avoiding the new added vertex $z_i$ in $G'_i/\{x_i,y_i\}$.

Combining the above two cases, if $G-(M_1 \cup M_2)$ is $2$-connected, then $M_1 \cup M_2$ is our desired matching, contradicting our assumption that $G-M$ is not $2$-connected for any such matching $M$.

It suffices to show that $G-(M_1 \cup M_2)$ is $2$-connected.  Suppose for contradiction that $u$ is a cut vertex in $G-(M_1 \cup M_2)$.

From the previous two cases, we know that $G'_i - M_i$ is still $2$-connected for each $i \in \{1,2\}$, and hence $G'_i - M_i - x_i$ is connected for $i = 1,2$. Because $M' = E_G[G_1,G_2] \neq \emptyset$, vertex $x$ is not a cut vertex of $G - M_1 - M_2$, which implies $u \neq x$. Without loss of generality, assume $u \in V(G_1)$. Let $H$ be a component of $G - (M_1 \cup M_2) - u$ that does not contain $x$. Then the vertex set $V(G_2)\cup\{x\}$, together with its neighborhood, lies in an another component of $G - (M_1 \cup M_2) - u$. Since $N_{G'_1-M_1}(x_1)\setminus\{u\} = N_G(V(G_2)\cup\{x\})$, it follows that $u$ is also a cut vertex of $G'_1 - M_1$. This contradicts the $2$-connectedness of $G'_1 - M_1$. Therefore, $G - M_1 - M_2$ is $2$-connected.

Combining the above two cases, since $G'_i-M_i$ remains $2$-connected for $i\in \{1,2\}$, $G-(M_1 \cup M_2)$ is $2$-connected.
Hence $G$ contains a $2$-removable $(\delta-3)$-matching (or $(\delta-2)$-matching if $\delta$ is even), contradicting our assumption that $G-M$ is not $2$-connected for any such matching $M$.
\end{pf}\hfill$\qed$

\begin{remark}
By Theorem \ref{2Connected(m-3)matching}, $f(2,\delta)\ge\begin{cases}
\delta-2, & \delta\ \text{is even};\\
\delta-3, & \delta\ \text{is odd}.
 \end{cases}$.
\end{remark}


\begin{remark}
Let $n, k$ be two integers. The join of an independent set of size $k-1$ and a cycle of order $n-k+1$ shows that $f(k,k+1)\le k$.
In particular, by Lemma \ref{2-R-2}, $f(2,3)=2$.
\end{remark}

Combining Theorem~\ref{thm:two-matching}, Theorem~\ref{thm:half-delta}, and the above results in this section, we have that Theorem \ref{thm:f-k-delta} holds.

\section{Conclusion}\label{sec:conclusion}

In this paper, we improve Halin's edge removability to matching removability in $k$-connected graphs with a minimum degree of at least $k+1$.

Based on our findings, we propose the following conjectures and an open problem, which we believe are worth exploring in future research:

\begin{conjecture}\label{con:matching}
For $k \ge 1$, every $k$-connected graph $G$ with $\delta(G) \ge k+1$ contains a $k$-removable $\lceil (\delta(G)+1)/2 \rceil$-matching, unless $\delta(G)$ is even and $G\cong K_{\delta(G)+1}$.
\end{conjecture}

Our Theorem~\ref{thm:half-delta} supports this conjecture for cases where $k\le 3$, with the notable exception of $k=3$ and $\delta(G)=4$, which suggests that further exploration may be required to fully understand this particular case.

For $f(k, \delta)$, we have the following problem.

\begin{problem}\label{pro:matching}
For positive integers $\delta > k \ge 1$, determine the exact value of $f(k, \delta)$.
\end{problem}

Theorem~\ref{thm:f-k-delta} gives exact values for $f(1,\delta)$ and $f(2,3)$. Determining $f(k,\delta)$ for larger $k$ and $\delta$ remains a compelling challenge and may reveal new structural properties of $k$-connected graphs.

We conclude with the following conjecture.

\begin{conjecture}\label{con:half-n-min}
Let $G$ be a $k$-connected graph on $n$ vertices with $\delta(G)\ge k+2$. Then $G$ contains a $k$-removable $\min\{\lfloor n/2\rfloor,\, \delta(G)\}$-matching.
\end{conjecture}

Theorem~\ref{thm:half-n-min} supports the conjecture for $k=1$. If Conjecture~\ref{con:half-n-min} holds, then $f(k,\delta)=\delta$ for $\delta\ge k+2$.

\end{document}